\documentclass[11pt, reqno]{amsart}
\usepackage{amsmath,amssymb,amscd,mathrsfs,amscd}
\usepackage{graphics,verbatim, tikz}
\usetikzlibrary{snakes,matrix,arrows,patterns}

\usetikzlibrary{matrix,arrows,patterns}
\usepackage{amsmath,amstext,amsbsy,amssymb,amscd}
\usepackage{amsmath}
\usepackage{amsxtra}
\usepackage{amscd}
\usepackage{amsthm}
\usepackage{amsfonts}
\usepackage{graphicx}%
\usepackage{amssymb}
\usepackage{eucal}
\usepackage{color}
\usepackage{multirow}
\usepackage[enableskew,vcentermath]{youngtab}
\usepackage[all]{xy}
\renewenvironment{quote}{\list{}{\leftmargin=8.5em \rightmargin=8.5em}\item[]}{\endlist}

\newtheorem{thm}{Theorem}[section]
\theoremstyle{plain}
\newtheorem{lem}[thm]{Lemma}
\newtheorem{prop}[thm]{Proposition}

\newtheorem{cor}[thm]{Corollary}

\newtheorem{dfn}[thm]{Definition}
\theoremstyle{definition}
\newtheorem{example}[thm]{Example}

\theoremstyle{remark}
\newtheorem{rem}[thm]{Remark}

\numberwithin{equation}{section}

\newcommand{\cA}{\mathcal A}
\newcommand{\A}{\mathcal A}

\newcommand{\C}{\mathbb C}

\newcommand{\Q}{\mathbb Q}
\newcommand{\Z}{\mathbb Z}

\newcommand{\var}{\varepsilon}

\newcommand{\la}{\lambda}

\newcommand{\ov}{\overline}
\newcommand{\osp}{\mathfrak{osp}}

{\vskip-\lastskip\medskip
  \noindent
  {\em #1.}\enspace
  }%
{\qed\par\medskip
  }

\newcommand{\Qv}{{\Q(v)}}
\newcommand{\Qvp}{{\Q(v)^\pi}}

\newcommand{\cB}{{\mathcal{B}}}

\newcommand{\height}{{\operatorname{ht}}}

\newcommand{\zero}{{\bar{0}}}
\newcommand{\one}{{\bar{1}}}

\newcommand{\te}{{\tilde{e}}}
\newcommand{\tf}{{\tilde{f}}}

\newcommand{\set}[1]{\left\{#1\right\}}
\newcommand{\parens}[1]{\left(#1\right)}
\newcommand{\ang}[1]{\left\langle#1\right\rangle}
\newcommand{\bra}[1]{\left[#1\right]}

\newcommand{\bbinom}[2]{\begin{bmatrix}#1 \\ #2\end{bmatrix}}

\renewcommand{\bar}[1]{\overline{#1}}

\newcommand{\chqg}{\mathbf{f}}      
\newcommand{\tpchqg}{{\mathbf{f}[t^{\pm 1}]}}
\newcommand{\schqg}{{\mathbf{f}[\bt]}} 
\newcommand{\cqg}{{\mathbf{U}}}      
\newcommand{\tpcqg}{{\mathbf{U}[t^{\pm 1}]}}
\newcommand{\cmqg}{{\dot{\mathbf{U}}}}      

\newcommand{\Qvtp}{{\Q(v)[t^{\pm 1}]^\pi}}

\newcommand{\Qvbtp}{{\Q(v,\bt)^\pi}}
\newcommand{\bt}{{\mathbf{t}}}
\newcommand{\bC}{{\mathcal C}}

\begin{document}
\title{Quantum supergroups III. Twistors}

\author[Clark, Fan, Li and Wang]{Sean Clark, Zhaobing Fan, Yiqiang Li and Weiqiang Wang}
\address{Department of Mathematics, University of Virginia, Charlottesville, VA 22904}
\email{
sic5ag@virginia.edu (S.Clark),
ww9c@virginia.edu (W. Wang)
}
\address{Department of Mathematics\\ University at Buffalo, SUNY 
\\Buffalo, NY 14260}
\email{
zhaobing@buffalo.edu (Z.Fan),
yiqiang@buffalo.edu (Y.Li)
}
\date{}

\maketitle

\begin{abstract}
We establish direct connections at several levels between quantum
groups and supergroups associated to bar-consistent anisotropic
super Cartan datum by constructing an automorphism (called twistor)
in the setting of covering quantum groups. The canonical bases of
the halves of quantum groups and supergroups are shown to match
under the twistor up to  powers of $\sqrt{-1}$. We further show that
the modified quantum group and supergroup  are isomorphic
 over the rational function field adjoined with  
 $\sqrt{-1}$, by constructing a twistor on the modified covering quantum group.
An equivalence of categories of weight modules for quantum
 groups and supergroups follows.
\end{abstract}


\vspace{3em}

\begin{quote}
{\em Le plus court chemin entre deux v\'{e}rit\'{e}s dans le domaine
r\'{e}el passe par le domaine complexe.}
\begin{flushright}{---Jacques Hadamard }
\end{flushright}
\end{quote}

\vspace{1em}

\section{Introduction}
\subsection{}

A theory of quantum supergroups was developed systematically by
Yamane \cite{Y1, Y2} after the classical work of Drinfeld, Jimbo and
Lusztig. Recently the interest in quantum supergroups has been
revived (see \cite{CW, CHW1, CHW2}) thanks to their categorification
\cite{HW} by Hill and one of the authors using the spin nilHecke and
quiver Hecke superalgebras \cite{W, EKL, KKT}. The work on quantum
supergroups of {\em anisotropic type} (meaning no isotropic odd
simple roots) has also motivated in turn further progress on
categorification. The conjecture in \cite{HW} that cyclotomic (spin)
quiver Hecke superalgebras categorify the integrable modules of the
supergroup has recently been proved by Kang, Kashiwara, and Oh
\cite{KKO13}. The validity of this conjecture at rank one, in which
case quiver Hecke superalgebras reduce to spin nilHecke algebras,
was  already noted in \cite{HW} as an easy upgrading of the
difficult categorification result of Ellis, Khovanov, and Lauda
\cite{EKL}. Yet another recent development is the categorification
of the modified covering quantum group in rank one
 (see Ellis-Lauda \cite{EL}).

A basic observation in \cite{HW} is that the parity functor $\Pi$
categorifies a formal super sign $\pi$ subject to $\pi^2=1$. This
leads to the formulation of the so-called covering quantum group
$\cqg$ in \cite{HW, CW, CHW1}, which allows a second formal
parameter $\pi$ such that $\pi^2=1$ besides the usual quantum
parameter $v$. The specialization of $\cqg$ at $\pi=1$, denoted by
$\cqg |_{\pi=1}$, recovers the usual quantum group while the
specialization of $\cqg$ at $\pi=-1$, denoted by  $\cqg |_{\pi=-1}$,
recovers the quantum supergroup of anisotropic type. In contrast to
the versions of quantum supergroups over $\C(v)$ studied in
literature, our covering (or super) quantum groups have a
well-developed representation theory such as weight modules and
integrable modules over $\Qv$, thanks to the enlarged Cartan
subalgebras \cite{CHW1}; moreover, they admit integral forms. Under
a mild bar-consistent condition on the super Cartan datum, the half
covering quantum group $\chqg$ ($\cong \cqg^{-}$) and the associated
integrable modules admit  a novel bar involution which sends $v
\mapsto \pi v^{-1}$ and then admit canonical bases \cite{CHW2}.

The (covering) quantum supergroups are quantizations of Lie
superalgebras associated to the anisotropic type super Cartan datum
introduced in \cite{Kac}. It has been known that Lie superalgebras
associated to the super Cartan datum  have representation theory
similar to that of Kac-Moody algebras associated to the same super
Cartan datum with $\Z_2$-grading forgotten; in particular, the
character formulas for the integrable modules of these Lie algebras
and superalgebras coincide. In the (only) finite type, this reduces
to the well-known fact that the finite-dimensional modules of Lie
superalgebra $\osp(1|2n)$ and Lie algebra $\mathfrak{so}(2n+1)$ have
the same characters. Such a similarity continues to hold at the
quantum level. But a conceptual explanation for all these
coincidences has been missing (see an earlier attempt \cite{La} in
finite type).

\subsection{}

The goal of this paper is to establish (somewhat surprising) direct links at several levels
between quantum groups and supergroups
associated to bar-consistent super Cartan datum, which
provide a conceptual explanation of the above coincidences.

We construct automorphisms (called twistors) denoted by $\Psi, \dot\Psi$ of
the half covering quantum group $\chqg$ and the modified  covering quantum group $\cmqg$,
respectively.
The construction of twistors requires
an extension of scalars to include a square root of $-1$, denoted by $\bt$ in this paper.
The twistor switches $\pi$ and $-\pi$, and hence specializes to an isomorphism between
the half (and resp., modified)  quantum group and its super counterpart.
As an immediate consequence, we obtain an equivalence of categories of weight modules for quantum
 group $\cqg|_{\pi=1}$ and supergroup $\cqg|_{\pi=-1}$.
 We also formulate an {\em extended covering quantum group} with enlarged Cartan subalgebra
 and construct its twistor.

Symbolically, we summarize the role of the twistor in the case of modified covering quantum group
in the following commutative diagram:

\begin{center}
\begin{tikzpicture}[scale=1]
\draw[->] (.15,1.3) arc (-50:220:.3);
\draw (-.05, 2.1) node {$\qquad\qquad \;\; \dot\Psi \; (\pi \mapsto -\pi)$};
\draw (.1,1) node {$\cmqg[\bt]$};
\draw (-2,-1) node {$\cmqg[\bt] |_{\pi=1}$};
\draw (2,-1) node {$\cmqg[\bt] |_{\pi=-1}$};

\draw[->>] (-.2,.7) -- (-1.6,-.6);
\draw[->>] (.2,.7) -- (1.6,-.6);
\draw[right hook->] (-1.9,-.6) -- (-.5,.7);
\draw[left hook->] (1.9,-.6) -- (.5,.7);
\draw[double,<->] (-1,-1) -- (1,-1) node[midway, above] {$\simeq$};
\end{tikzpicture}
\end{center}
Alternatively, one can view the modified quantum group $\cmqg
|_{\pi=1}$ and the modified quantum supergroup $\cmqg |_{\pi=-1}$ as
two different rational forms of a common algebra $\cmqg[\bt]
|_{\pi=1}$. The two rational forms admit their own distinct integral
forms. Remarkably the distinction between super  and non-super
algebras becomes blurred at the quantum level, even though a clear
distinction exists between Lie algebras and Lie superalgebras (for
example, there are ``more" integrable modules for Lie algebras than
for the corresponding Lie superalgebras \cite{Kac}).

 As an application, the twistor $\Psi$ induces a transformation
on the crystal lattice of $\chqg$ which behaves well with the crystal structure.
By careful bookkeeping, we provide a purely algebraic proof
of \cite[Proposition~6.7]{CHW2} that the crystal
lattice of $\chqg$ is invariant
under an anti-automorphism $\varrho$ which fixes the Chevalley generators.
Furthermore,  the twistor $\Psi$ is shown to match
Lusztig-Kashiwara's canonical basis for $\chqg|_{\pi=1}$ \cite{Lu1, K} with
the canonical basis for the half quantum supergroup $\chqg|_{\pi=-1}$ constructed in \cite{CHW2}, up to
integer powers of  $\bt$. Let us add that this does not give a new proof of the existence
of the canonical basis for $\chqg$ or for the integrable modules of $\cqg$.

\subsection{}

Although it is not very explicitly used in this paper, the
connection between (one-parameter) quantum groups and two-parameter
$(v,t)$-quantum groups developed by two of the authors \cite{FL12}
plays a basic role in our evolving understanding of the links
between quantum groups and supergroups. A connection between
(one-parameter) quantum groups and quantum supergroups can indeed be
formulated by a ``twisted lift" to two-parameter quantum groups
which is followed by a ``specialization" of the second parameter $t$
to $\bt$ with $\bt^2=-1$. But we have decided to adopt the more
intrinsic and self-contained approach as currently formulated in
this paper.

The isomorphism result on modified quantum (super)groups $\cmqg[\bt]
|_{\pi=1}$ and $\cmqg[\bt] |_{\pi=-1}$ in this joint work was
announced in \cite{FL13},  where the isomorphism in the rank one
case was established somewhat differently from here.

A version of our equivalence of categories of weight modules for
$\cqg|_{\pi=1}$ and $\cqg|_{\pi=-1}$ also appeared in \cite{KKO13}
with a very different proof. Note that the notion of weight modules
in {\em loc. cit.}  is nonstandard and subtle, and the
multi-parameter algebras formulated therein over $\C(v)$ or
$\C(v)^\pi$ do not seem to admit rational forms or integral forms or
modified counterparts as ours. Some construction similar to the
twistor $\widehat{\Psi}$ for our extended covering quantum group
(see Proposition~\ref{prop:qgiso})  also appeared in \cite{KKO13}.
In contrast to {\em loc. cit.}, our formula for $\widehat{\Psi}$ is
very explicit; the twistor $\dot\Psi$ here preserves the integral
forms (see Theorem~\ref{thm:modauto}), and this allows us to
specialize $v$ to be a root of unity without difficulty.

\subsection{}

The paper is organized as follows.

In Section~\ref{sec:half}, after recalling some preliminaries,
we formulate and establish a  twistor $\Psi$ of the half covering quantum group
$\chqg[\bt]$,
which restricts to an isomorphism between the super and non-super half quantum groups.
Here we make a crucial use of a new multiplication on $\chqg[\bt]$ twisted by a distinguished bilinear form,
and the general idea of such twisted multiplication goes back to \cite{FL12}.

In Section~\ref{sec:CB}, we use the twistor $\Psi$ to compare the
crystal lattices between the $\pi=1$ and $\pi=-1$ cases. In
particular, we give an algebraic proof that the crystal lattice for
$\chqg$ is preserved by an anti-involution $\varrho$. (This was
stated in \cite[Proposition~6.7]{CHW2}.) Then we show that the
twistor $\Psi$ matches the canonical basis elements of the half
quantum supergroup $\chqg |_{\pi=-1}$ and those of half quantum
group $\chqg |_{\pi=1}$, up to integer powers of $\bt$.

In Section~\ref{sec:modified}, we construct a twistor of the
modified covering quantum group. This restricts to an isomorphism
between the super and non-super modified quantum groups. An
immediate corollary is an equivalence of categories of weight
modules for the super and non-super quantum groups. A further
consequence is an equivalence of BGG categories of modules for
Kac-Moody Lie algebras and Lie superalgebras. Finally we construct
an alternative twistor relating quantum groups to quantum
 supergroups upon enlarging the Cartan subalgebras.

\vspace{.2cm} \noindent \textbf{Acknowledgements.} Y.L. is supported
in part by the NSF grant DMS-1160351, while S.C. and W.W. are
partially supported by the NSF grant DMS-1101268. S.C. was also
supported by a semester fellowship at University of Virginia. S.C.
and W.W. thank Institute of Mathematics, Academia Sinica, Taipei for
providing an excellent working environment and support, where part
of this project was carried out. W.W. thanks Shun-Jen Cheng and
Maria Gorelik for helpful discussions regarding the work of
Lanzman.

\section{The twistor of half covering quantum group}
 \label{sec:half}

\subsection{The preliminaries}
We review some basic definitions which can be found in \cite{CHW1, CHW2} and references therein.

\begin{dfn}
  \label{dfn:scd}
A {\em Cartan datum} is a pair $(I,\cdot)$ consisting of a finite
set $I$ and a $\Z$-valued symmetric bilinear form $\nu,\nu'\mapsto \nu\cdot\nu'$
on the free abelian group $\Z[I]$  satisfying
\begin{enumerate}
 \item[(a)] $d_i=\frac{i\cdot i}{2}\in \Z_{>0}, \quad \forall i\in I$;

  \item[(b)]
$a_{ij}=2\frac{i\cdot j}{i\cdot i}\in \Z_{\leq 0}$,  for $i\neq j$ in $I$.
\end{enumerate}
A Cartan datum is called a {\em super Cartan datum of anisotropic
type} if there is a partition $I=I_\zero\coprod I_\one$ which
satisfies the condition
\begin{enumerate}
        \item[(c)] $2\frac{i\cdot j}{i\cdot i} \in 2\Z$ if $i\in I_\one$ and $j \in I$.
\end{enumerate}
A super Cartan datum of anisotropic type is called  {\em
bar-consistent}
if it additionally satisfies
\begin{enumerate}
        \item[(d)]  $d_i\equiv p(i) \mod 2, \quad \forall i\in I.$
\end{enumerate}
\end{dfn}
We will always assume $I_\one\neq\emptyset$ without loss of
generality. We note that (d)  is almost always satisfied for super
Cartan data of finite or affine type (with one exception which
corresponds to a Dynkin diagram with two short roots of opposite
parity at its both ends, called by $A^{(4)}(0,2n)$). A super Cartan
datum is always assumed to be bar-consistent in this paper.
We note that a bar-consistent super Cartan datum satisfies
\begin{equation}  \label{eq:even}
i\cdot j\in 2\Z \quad \text{ for all }i,j\in I.
\end{equation}

The $i\in I_\zero$ are called even, $i\in I_\one$ are called odd. We
define a parity function $p:I\rightarrow\set{0,1}$ so that $i\in
I_{\overline{ p(i)}}$. We extend this function to the homomorphism
$p:\Z[I]\rightarrow \Z_2$. Then $p$ induces a {\em parity $\Z_2$-grading} on
$\Z[I]$.
We define the height function $\height$ on $\Z[I]$
by letting $\height(\sum_{i\in I} c_i i)=\sum_{i\in I} c_i$.

A {\em super root datum} associated to a super Cartan datum
$(I,\cdot)$ consists of
\begin{enumerate}
\item[(a)]
two finitely generated free abelian groups $Y$, $X$ and a
perfect bilinear pairing $\ang{\cdot, \cdot}:Y\times X\rightarrow \Z$;

\item[(b)]
an embedding $I\subset X$ ($i\mapsto i'$) and an embedding $I\subset
Y$ ($i\mapsto i$) satisfying

\item[(c)] $\ang{i,j'}=\frac{2 i\cdot j}{i\cdot i}$ for all $i,j\in I$.
\end{enumerate}

We will assume that the image of the imbedding $I\subset X$
(respectively, the image of the imbedding $I\subset Y$) is linearly
independent in $X$ (respectively, in $Y$);
in the terminology of \cite{L93}, this means the datum is both $X$-regular and
$Y$-regular.

If $V$ is a vector space graded by $\Z[I]$, $X$, or $Y$, we will use the weight notation
$|x|=\mu$ if $x\in V_\mu$. If $V$ is a $\Z_2$-graded vector space, we will
use the parity notation $p(x)=a$ if $x\in V_a$.


Let $v$ and $t$ be formal parameters, and let $\pi$ be an indeterminate
such that
$$
\pi^2=1.
$$
For a ring $R$ with $1$, we will form a new ring
$R^\pi=R[\pi]/(\pi^2-1)$. Given an $R^\pi$-module (or algebra) $M$,
the {\em specialization of $M$ at $\pi=\pm 1$} means the $R$-module
(or algebra) $M|_{\pi=\pm 1} \stackrel{\text{def}}{=}R_{\pm}\otimes_{R^\pi} M$,
where $R_\pm =R$ is viewed
as a $R^\pi$-module on which $\pi$ acts as $\pm 1$.

Assume 2 is invertible in $R$; i.e. $\frac{1}{2}\in R$.
We define
\begin{equation}\label{eq:pi idempotent}
\var_{+}=\frac{1+ \pi}{2},\qquad\var_{-}=\frac{1- \pi}{2},
\end{equation}
and note that $R^\pi=R\var_+\oplus R\var_-$.
In particular, since $\pi \var_{\pm}=\pm \var_{\pm}$
for an $R^\pi$-module $M$, we see that
\[M|_{\pi=\pm 1}\cong \var_{\pm } M.\]

Similarly, for an $R$-module $M$, we define
\[M[t^{\pm 1}]=R[t^{\pm 1}]\otimes_R M.\]
Let $\bt^2=-1 \in R$. Let us define the specialization of $t$ at $\bt$
to be
\[
M[\bt]=R[\bt]\otimes_{R[t^{\pm 1}]} M[t^{\pm 1}]=R[\bt]\otimes_R M.
\]
(Note that the results herein may be reformulated in a context where $\bt$ is replaced by an
indeterminate solution to the equation $t^4=1$.)


Recall $\pi^2=1$.  For $k \in \Z_{\ge 0}$ and $n\in \Z$,
we introduce a $(v,\pi)$-variant of quantum integers, quantum factorial and quantum binomial coefficients:

\begin{equation}
 \label{eq:nvpi}
\begin{split}
\bra{k}_{v,\pi} & 
=\frac{(\pi v)^k-v^{-k}}{\pi v-v^{-1}}  \in \Z[v^{\pm 1}]^\pi,
  \\
\bra{k}_{v,\pi}^!  &= \prod_{l=1}^k \bra{l}_{v,\pi}   \in \Z[v^{\pm 1}]^\pi,
 \\
\bbinom{n}{k}_{v,\pi} &=\frac{\prod_{l=n-k+1}^n  \big( (\pi v)^{l}
-v^{-l} \big)}{\prod_{l=1}^k \big( (\pi v)^{l}- v^{-l} \big)}  \in
\Z[v^{\pm 1}]^\pi.
\end{split}
\end{equation}

We will use the notation
$$
v_i=v^{d_i}, \quad t_i=t^{d_i}, \quad \pi_i=\pi^{d_i}, \quad \text{ for } i\in I.
$$

Let $(I,\cdot)$ be a super Cartan datum.
The {\em half covering quantum group} $\chqg$ \cite[\S 1]{CHW1} is the $\Qvp$-algebra with generators
$\theta_i$ ($i\in I$) and relations
\begin{equation}\label{eq:serrerel}
\sum_{k=0}^{b_{ij}} (-1)^k\pi^{\binom{k}{2}p(i)+kp(i)p(j)}\bbinom{b_{ij}}{k}_{v_i, \pi_i}
 \theta_i^{b_{ij}-k}\theta_j\theta_i^k=0 \quad \text{for all }i\neq j\in I,
\end{equation}
where
$$
b_{ij} = 1- a_{ij}.
$$
As first noted in \cite{HW},
the $\Q$-algebra $\chqg$ admits a bar involution $\bar{\phantom{c}}$ such that
\begin{equation}
 \label{eq:bar}
\ov{\theta_i} =\theta_i\; (\forall i\in I), \qquad \ov{\pi} =\pi, \qquad \ov{v} =\pi v^{-1}.
\end{equation}
We define the divided powers
\begin{equation}\label{eq:thetadivpow}
\theta_i^{(n)}=\frac{\theta_i^n}{\bra{n}^!_{v_i,\pi_i}}.
\end{equation}
These elements generate a $\Z[v^{\pm 1}]^\pi$-subalgebra
of $\chqg$,  denoted by ${}_\Z\chqg$. (In this paper, the notation $\Z[v^{\pm 1}]$ stands
for the ring of Laurent polynomias in $v$.) Note that $\theta_i^{(n)}$ is bar invariant.

By specialization at $\pi=\pm 1$, we obtain the usual half quantum
group $\chqg|_{\pi=1}$ and the half quantum supergroup
$\chqg|_{\pi=-1}$, respectively. By leaving $\pi$ as an
indeterminate, we can simultaneously address both cases.

The algebra $\chqg$ has a $\Z[I]\times \Z_2$-grading obtained by setting $|\theta_i|=i$
and $p(\theta_i)=p(i)$, for $i \in I$.
The algebra $\chqg$ is known \cite{HW, CHW1}
to be equipped with a nondegenerate symmetric bilinear form $(\cdot,\cdot)$
such that
\[(1,1)=(\theta_i,\theta_i)=1,\quad (\theta_ix,y)=(x,e_i'(y)),\]
where $e_i':\chqg\rightarrow \chqg$ is the map satisfying
\begin{equation}  \label{eq:e'}
e_i'(1)=0, \quad e_i'(\theta_j)=\delta_{ij},\quad e_i'(xy)=e_i'(x)y+\pi^{p(i)p(x)}v^{-i\cdot |x|} xe_i'(y).
\end{equation}
There exists \cite{CHW2} a (non-super) algebra anti-automorphism of $\chqg$ such that
\begin{equation}   \label{eq:rho}
\varrho(\theta_i)=\theta_i\; (\forall i\in I),
\qquad \varrho(xy)=\varrho(y)\varrho(x), \quad \forall x,y \in \chqg.
\end{equation}


\subsection{A twisted multiplication}
 \label{subsec:twist}

Fix once and for all a total order $<$ on $I$.
Recall the notation $d_i, a_{ij}$ from Definition~\ref{dfn:scd}.
Let $\phi:\Z[I]\times \Z[I]\rightarrow \Z$
be the bilinear form defined by: for $i,j \in I$,
\begin{equation}\label{eq:phidef}
\phi(i,j)=\begin{cases}
d_ia_{ij}&\text{ if } j<i,\\
d_i &\text{ if } j=i,\\
-2p(i)p(j)&\text{ if } j>i.\\
\end{cases}
\end{equation}
Set
\begin{equation*}
\delta_{i<j} =\begin{cases}
0, & \text{ if } i\not < j,
\\
1, & \text{ if } i<j.
\end{cases}
\end{equation*}
By abuse of notation we regard $\Z_2 =\{0,1\} \subset \Z$, and so by \eqref{eq:even} we have
\begin{equation*}
\phi(i,j)- \phi(j,i) =(-1)^{\delta_{i<j}} \big(i\cdot j +   2p(i)p(j) \big) \in 2\Z,\qquad \text{for }i\neq j.
\end{equation*}
In particular, we always have
\begin{equation}   \label{eq:phisymmetrized}
\phi(i,j)- \phi(j,i) \equiv  i\cdot j +   2p(i)p(j) \mod 4,\qquad \text{for }i\neq j.
\end{equation}

Recall that $\tpchqg$ denotes the $\Qvtp$-algebra $\Qvtp\otimes_{\Qvp}\chqg$.
Define a new multiplication $*$ on $\tpchqg$ by setting
\begin{equation}  \label{eq:x*y}
x*y=t^{\phi(|x|,|y|)} xy,
\end{equation}
for homogeneous $x,y\in \tpchqg$ and then extending it bilinearly.
Since $\phi$ is bilinear, one verifies that $(\tpchqg, *)$ is a $\Z[I]$-graded associative algebra
generated by $\theta_i$. We will use the notation $x^{*n}=\underbrace{x*x*\ldots*x}_{n}$
for powers taken with respect to this product.
We note that
\begin{equation}\label{eq:varrho*}
\varrho(x*y)=t^{\phi(|x|,|y|)-\phi(|y|,|x|)}\varrho(y)*\varrho(x), \quad \forall x, y\;\mbox{homogeneous}.
\end{equation}

\begin{prop}
The algebra $(\tpchqg,*)$ has a presentation as the $\Q(v)[t^{\pm 1}]^\pi$-algebra
with generators $\theta_i$ $(i\in I)$ and relations
\begin{equation}
 \label{eq:Serre*}
\sum_{k=0}^{b_{ij}}  (-1)^k \pi^{\binom{k}{2}p(i)+kp(i)p(j)}
 t^{k(b_{ij}-k)d_i + (b_{ij}-k)\phi(i,j)+k\phi(j,i)}
\bbinom{b_{ij}}{k}_{v_i,\pi_i}
 \!\!\theta_i^{*\, b_{ij}-k}*\theta_j*\theta_i^{*k}=0,
\end{equation}
for all $i\neq j\in I$.
\end{prop}

\begin{proof} 
The relation \eqref{eq:Serre*} for $(\tpchqg,*)$ can be derived directly from the
Serre relation \eqref{eq:serrerel} for $\chqg$, and vice versa. As the computation is straightforward,
we skip the details.
\end{proof}

\begin{rem}
The twisted $*$-product on $\tpchqg$ is a variant of the transformation defined
in \cite[\S 4]{FL12} to relate one-parameter quantum group to two-parameter quantum group.
The precise formula for the bilinear form $\phi$ is new, and it plays a crucial role in this paper.
\end{rem}

\subsection{The twistor $\Psi$}  

Recall that we set $\bt^2=-1$ and
that $\schqg$ is the $\Qvbtp$-algebra $\Qvbtp\otimes_{\Qvtp}\tpchqg$.
By specializing $t$ and twisting $v$, we obtain
the following $\Q(\bt)$-algebra isomorphism which plays a fundamental role in this paper.

\begin{thm}\label{thm:halfiso}
There is a $\Q(\bt)$-algebra isomorphism $\Psi:\chqg[\bt] \rightarrow (\chqg[\bt], *)$
satisfying
\begin{equation}  \label{eq:psi}
\Psi(\theta_i)=\theta_i\, (i\in I), \quad \Psi(v)=\bt^{-1} v, \quad \Psi(\pi)=-\pi,\quad \Psi(xy)=\Psi(x)*\Psi(y).
\end{equation}
\end{thm}
The transformation $\Psi$ is called the {\em twistor} on $\schqg$.

\begin{proof}
Set
\[
S_{ij} :=\sum_{k=0}^{b_{ij}} (-1)^k (-\pi)^{\binom{k}{2}p(i)+kp(i)p(j)}
 \bbinom{b_{ij}}{k}_{\bt_i^{-1} v_i,(-\pi)_i}
 \theta_i^{*\, b_{ij}-k}*\theta_j*\theta_i^{*k}.
\]
To show such a $\Q(\bt)$-linear map $\Psi$ exists, it suffices to show that
the images of the generators satisfy \eqref{eq:serrerel} with respect to $*$;
that is,
\begin{equation}  \label{eq:Sij}
S_{ij} =0 \quad \text{for all }i\neq j\in I.
\end{equation}
To that end, fix $i\neq j\in I$. Unraveling the definition of $*$, we have
\begin{align*}
S_{ij}=\sum_{k=0}^{b_{ij}} &(-1)^k(-\pi)^{\binom{k}{2}p(i)+kp(i)p(j)} \bbinom{b_{ij}}{k}_{\bt_i^{-1} v_i,(-\pi)_i}\\
 &\times \bt^{(\binom{k}{2}+\binom{b_{ij}-k}{2}+k(b_{ij}-k))d_i
 + (b_{ij}-k)\phi(i,j)+k\phi(j,i)}\theta_i^{b_{ij}-k}\theta_j\theta_i^{k}.
\end{align*}
One verifies that $\binom{k}{2}+\binom{b_{ij}-k}{2}=\binom{b_{ij}}{2}-k(b_{ij}-k)$
and $\bbinom{b_{ij}}{k}_{\bt_i^{-1} v_i,(-\pi)_i}
=\bt^{k(b_{ij}-k)d_i}\bbinom{b_{ij}}{k}_{v_i,\pi_i}$.
Using these identities, we rewrite the above identity for $S_{ij}$ as
\begin{align}
\bt^{-\binom{b_{ij}}{2}d_i} S_{ij}
&= \sum_{k=0}^{b_{ij}} (-1)^k(-\pi)^{\binom{k}{2}p(i)+kp(i)p(j)}\bbinom{b_{ij}}{k}_{\bt_i^{-1} v_i,(-\pi)_i}
 \bt^{(b_{ij}-k)\phi(i,j)+k\phi(j,i)}\theta_i^{b_{ij}-k}\theta_j\theta_i^{k}
   \notag \\
&= \sum_{k=0}^{b_{ij}} (-1)^k(-\pi)^{\binom{k}{2}p(i)+kp(i)p(j)}\bbinom{b_{ij}}{k}_{v_i,\pi_i}
  \bt^{\clubsuit}\theta_i^{b_{ij}-k}\theta_j\theta_i^{k},
    \label{eq:tS}
\end{align}
where
\begin{equation}  \label{eq:club}
\clubsuit=k(b_{ij}-k)d_i + (b_{ij}-k)\phi(i,j)+k\phi(j,i).
\end{equation}

Now let us consider $\clubsuit$.
First assume that $i<j$.
Then we find that
\begin{align*}
\clubsuit
&= k(b_{ij}-k)d_i - 2(b_{ij}-k)p(i)p(j)+kd_ia_{ij}
  \\
&= -2\binom{k}{2}d_i+2kp(i)p(j)-2b_{ij}p(i)p(j).
\end{align*}
Next assume that $i>j$.
Then we have
\begin{align*}
\clubsuit
&= k(b_{ij}-k)d_i + (b_{ij}-k)d_ia_{ij}-2kp(i)p(j)
 \\
 &=
 -2\binom{k}{2}d_i+a_{ij}(b_{ij}-2k)d_i-2kp(i)p(j).
\end{align*}
Note that $2a_{ij} d_i \equiv 0 \mod 4$, thanks to \eqref{eq:even}.
In either case when $i<j$ or $i>j$, we see that
\[
\clubsuit= 2\binom{k}{2} d_i+2kp(i)p(j)+c(i,j) \mod 4,
\]
where
\begin{equation*}
c(i,j) =
\begin{cases}
2b_{ij}p(i)p(j), & \text{ if }i<j,
 \\ \\
-d_i \binom{a_{ij}}{2}, & \text{ if } i>j.
\end{cases}
\end{equation*}
Recall $\bt^2=-1$. By the bar-consistent condition we have $2d_i=2p(i) \mod 4$,
and thus
$\bt^\clubsuit=\bt^{c(i,j)}(-1)^{\binom{k}{2}p(i)+kp(i)p(j)}$.
Then we can rewrite \eqref{eq:tS}
and apply the Serre relation \eqref{eq:serrerel} for $\chqg$ to conclude that
\[\bt^{-\binom{b_{ij}}{2}d_i-c(i,j)} S_{ij}=\sum_{k=0}^{b_{ij}}
 (-1)^k\pi^{\binom{k}{2}p(i)+kp(i)p(j)}
 \bbinom{b_{ij}}{k}_{v_i,\pi_i}
 \theta_i^{b_{ij}-k}\theta_j\theta_i^{k}=0.
 \]
Therefore, \eqref{eq:Sij} is verified and $\Psi$ is well defined.

Finally, to see that $\Psi$ is an isomorphism, we note that
a similar argument can be used to  show that a map $\Phi: (\schqg, *) \rightarrow \schqg$ satisfying
\[\Phi(\theta_i)=\theta_i, \quad \Phi (v)=\bt v, \quad \Phi(\pi)=-\pi,\quad \Phi(x*y)=\Phi(x)\Phi(y),\]
is well defined as well; clearly $\Phi$ is the inverse of $\Psi$.
\end{proof}

Theorem~\ref{thm:halfiso} provides a way to compare the super and non-super half quantum groups via
$\Psi$. Indeed, recall the idempotents $\var_{\pm}$
from \eqref{eq:pi idempotent}. Then from
$
\Psi(\pi)=-\pi
$
, we see that
$
\Psi(\var_{\pm})=\var_{\mp}
$
. In particular, $\Psi(\var_{\pm}\schqg)=\var_{\mp}\schqg$, in effect
swapping the super and non-super specializations at $\pi=-1$ and $\pi=1$.

\begin{cor}\label{cor:half super to non}
There is a $\Q(\bt)$-linear isomorphism $\Psi:\schqg|_{\pi=1}\rightarrow \schqg|_{\pi=-1}$.
\end{cor}

Using the identification $\schqg|_{\pi=\pm 1}\cong \var_{\pm}\schqg$,
we have inclusions $\schqg|_{\pi=\pm 1}\hookrightarrow \schqg$.
Theorem~\ref{thm:halfiso} and Corollary~\ref{cor:half super to non} can be summarized symbolically in the following
diagram:

\begin{center}
\begin{tikzpicture}[scale=1]
\draw[<->] (-1.6,1.) -- (1.4,1.) node[midway,above] {$\Psi$};
\draw (-2,1) node {$\chqg[\bt]$};
\draw (2.3,1) node {$(\chqg[\bt],*)$};
\draw (-2,-1) node {$\chqg[\bt] |_{\pi=1}$};
\draw (2.3,-1) node {$\chqg[\bt] |_{\pi=-1}$};

\draw[right hook->] (-1.9,-.6) -- (-1.9,.6);
\draw[<<-] (2,-.6) -- (2,.6);
\draw[right hook->] (2.4,-.6) -- (2.4,.6);
\draw[<<-] (-2.3,-.6) -- (-2.3,.6);
\draw[snake=snake,-] (-.9,-1) -- (.9,-1);
\draw[<-] (-1.2,-1) -- (-.9,-1);
\draw[<-] (1.4,-1) -- (.9,-1);
\end{tikzpicture}
\end{center}

For $i_1,\ldots, i_n \in I$, we denote
\begin{align*}
\mathbf{N}(i_1+\ldots +i_n) &=\sum_{1\leq r<s\leq n} i_r\cdot i_s,
  \\
\mathbf{p}(i_1+\ldots +i_n) &=\sum_{1 \leq r<s \leq n}p(i_r)p(i_s).
\end{align*}
By convention, $\mathbf{N}(i_1)=\mathbf{p}(i_1)=0$.
Note that $\mathbf{N}(\cdot)$ is always an even integer by \eqref{eq:even}.

The following proposition on the $\Q(\bt)$-linear involution
$\varrho$ of $\schqg$ will be used in the next section.

\begin{prop}  \label{prop:rhopsi}
The involutions $\Psi\varrho\Psi^{-1}$ and $\varrho$ on $\schqg$ are
equal up to a sign on each weight space. More precisely, we have
\begin{equation}
  \label{eq;rhopsi}
\Psi\varrho\Psi^{-1}(x) =(-1)^{\frac{\mathbf{N}(\nu)}{2}+\mathbf{p}(\nu)} \varrho(x), \quad \text{ for } x \in \chqg_\nu.
\end{equation}
\end{prop}

\begin{proof}
We prove the formula \eqref{eq;rhopsi} by induction on the height $\height(|x|)$.

The formula clearly holds when $\height(|x|) \le 1$.

Now assume that the formula holds for $x$ with $\height (|x|)\ge 1$
and for $y$ with $\height(|y|)\ge 1$. Recall $\bt^2=-1$.
Then by \eqref{eq:rho}, \eqref{eq:phisymmetrized}, \eqref{eq:x*y},  \eqref{eq:varrho*}, and \eqref{eq:psi},  we have
\begin{align*}
\Psi \varrho \Psi^{-1} (x*y)
&= \Psi \big( \varrho (\Psi^{-1}(y))\, \varrho (\Psi^{-1}(x) ) \big)
  \\
&= \Psi \varrho \Psi^{-1}(y) * \Psi \varrho \Psi^{-1}(x)
  \\
&= (-1)^{\frac{\mathbf{N}(|y|)}{2}+\mathbf{p}(|y|)+\frac{\mathbf{N}(|x|)}{2}+\mathbf{p}(|x|)}
\varrho  (y)  * \varrho  (x)
  \\
&= (-1)^{\frac{\mathbf{N}(|y|)}{2}+\mathbf{p}(|y|)+\frac{\mathbf{N}(|x|)}{2}+\mathbf{p}(|x|)}
\bt^{\phi(|y|,|x|) -\phi(|x|,|y|)}  \varrho  (x*y)
  \\
&= (-1)^{\frac{\mathbf{N}(|x*y|)}{2}+\mathbf{p}(|x*y|)}\varrho(x*y).
\end{align*}
Hence the formula \eqref{eq;rhopsi} holds for $x*y$.  This completes the induction.

Since $\mathbf{N}$ and $\mathbf{p}$ only depend on the weight,
$\Psi\varrho\Psi^{-1}$ and $\varrho$ are proportional on each weight space.
The proposition is proved.
\end{proof}

\section{Comparison of crystal lattices and canonical bases}
 \label{sec:CB}


\subsection{Comparing crystal lattices}

For $x\in \chqg_\nu$, there is a unique decomposition of the form
\begin{equation}  \label{eq:string}
x=\sum_{n\geq 0} \theta_i^{(n)}x_n,
\end{equation}
such that $x_n=0$ for all but finitely many $n$,
$x_n\in \chqg_{\nu- ni}$, and $e_i'(x_n)=0$ for all $n$.
We will refer to this as its {\em $\bf i$-string decomposition}.
Then we define Kashiwara operators
\[\te_i x=\sum_{n\geq 1} \theta_{i}^{(n-1)}x_n,\]
\[\tf_i x=\sum_{n\geq 0} \theta_{i}^{(n+1)}x_n.\]

Let $\A\subset \Q(v)$ be the ring of rational functions with no poles at $v=0$
and so $\A^\pi=\A[\pi]\subset \Qvp$.
The crystal lattice ${\mathcal L}$ of $\chqg$ is the $\A^\pi$-lattice
generated by \[B=\set{\tf_{i_1}\ldots \tf_{i_n} 1\mid \forall  i_1,\ldots, i_n\in I, \forall n}.\]
According to \cite{CHW2},
the set $\cB :=(B\cup \pi B)+v{\mathcal L}$ is a $\Q$-basis of ${\mathcal L}/v{\mathcal L}$,
called the (maximal) crystal basis for $\chqg$.

We note the following useful properties of ${\mathcal L}$
(with the same proof as usual \cite{K}).

\begin{lem}\label{lem:latticefacts}
Let $x=\sum_{n\ge 0} \theta_i^{(n)} x_n$ be the $i$-string decomposition
of $x\in \chqg$. Then,
\begin{enumerate}
\item $x\in {\mathcal L}$ if and only if $x_n\in {\mathcal L}$
for all $n$.
\item If $x+v{\mathcal L}\in \cB$, then
$x=\theta_i^{(n)} x_n$ mod $v{\mathcal L}$ for some $n$
and $x_n+v{\mathcal L}\in \cB$.
\item If $\te_j x=0$ for all $j\in I$ then $x=0$; if $\te_j x\neq 0$ then
$\tf_j \te_j x=x$.
\end{enumerate}
\end{lem}

To take advantage of Theorem~\ref{thm:halfiso}, we need to extend
scalars to include $\bt$. We let $\cA[\bt]=\Q(\bt)\otimes_\Q \cA$, the subring of $\Q(v,\bt)$
of rational functions with no poles at $v=0$. Then
set ${\mathcal L}[\bt]=\cA[\bt]^\pi\otimes_{\cA^\pi} {\mathcal L}$.

The isomorphism $\Psi$ in Theorem \ref{thm:halfiso}, which sends $v\mapsto \bt^{-1} v$
and $\pi\mapsto -\pi$, clearly preserves the $\Q(\bt)$-algebra $\A[\bt]^\pi$.

\begin{lem}\label{lem:psilattice} The following properties hold:
\begin{enumerate}
\item
$\Psi(\theta_i^{(n)})=\theta_i^{(n)}$ for $n\ge 1$;

\item $e_i'(\Psi(x))=\bt^{\phi(i,|x|-i)}\Psi(e_i'(x))$ for all homogeneous
$x\in \schqg$ and $i\in I$;

\item
Let $x\in \schqg_\nu$ with its $i$-string decomposition \eqref{eq:string} for a given $i\in I$.
Then $\Psi(x)$ has the following $i$-string decomposition
\[\Psi(x)=\sum_{n\geq 0} \bt^{\phi(ni,\nu)-n^2d_i} \theta_i^{(n)}\Psi(x_n).\]
\end{enumerate}
\end{lem}

\begin{proof}
Recall the definitions \eqref{eq:nvpi} of $\bra{n}_{v,\pi}$ and \eqref{eq:psi} of $\Psi$. We have
\[
\Psi \big( \bra{n}_{v,\pi} \big) = \bra{n}_{\bt^{-1} v,-\pi}=\bt^{n-1}\bra{n}_{v,\pi}.
\]
We prove (1) by induction on $n$. The case when $n=1$ is  clear. Assume
$\Psi(\theta_i^{(n-1)})=\theta_i^{(n-1)}$. By definition of the divided power \eqref{eq:thetadivpow}, we have
\begin{align*}
\Psi(\theta_i^{(n)}) &= \Psi \big(\bra{n}_{v_i,\pi_i}^{-1} \theta_i  \theta_i^{(n-1)} \big)
 \\
 &=  \bt_i^{1-n}  \bra{n}_{v_i,\pi_i}^{-1} \Psi(\theta_i) * \Psi \big( \theta_i^{(n-1)} \big)
 \\
 &=  \bt_i^{1-n}  \bra{n}_{v_i,\pi_i}^{-1}  \bt_i^{n-1}  \theta_i  \theta_i^{(n-1)}
 =\theta_i^{(n)}.
\end{align*}

Now let us verify (2). It is trivial if $\height |x|\leq 1$.
Otherwise, it suffices to show that if (2) holds for $x,y\in \schqg$,
then it holds for $xy$.
By \eqref{eq:e'}  we compute
\begin{align*}
e_i'(\Psi(xy))&=\bt^{\phi(|x|,|y|)}e_i'(\Psi(x)\Psi(y))
 \\
&=\bt^{\phi(|x|,|y|)} \big( e_i'(\Psi(x))\Psi(y) +\pi^{p(i)p(x)}v^{-i\cdot |x|} \Psi(x) e_i'(\Psi(y)) \big)
 \\
&=\bt^{\phi(i,|y|)}e_i'(\Psi(x))*\Psi(y)+\pi^{p(i)p(x)}v^{- i\cdot |x|}
\bt^{\phi(|x|,i)}\Psi(x)*e_i'(\Psi(y))
 \\
&\stackrel{(\star)}{=}\bt^{\phi(i,|y|)+\phi(i,|x|-i)}\Psi(e_i'(x)y)+ \pi^{p(i)p(x)} v^{-i\cdot |x|}
\bt^{\phi(|x|,i)+\phi(i,|y|-i)}\Psi(xe_i'(y))
 \\
&\stackrel{(\star\star)}{=}\bt^{\phi(i,|y|)+\phi(i,|x|-i)}\Psi(e_i'(x)y)+(-\pi)^{p(i)p(x)}(\bt^{-1} v)^{- i\cdot |x|}
\bt^{\phi(i,|x|)+\phi(i,|y|-i)}\Psi(xe_i'(y))
 \\
&=\bt^{\phi(i,|x|+|y|-i)}\Psi \big(e_i'(x)y+\pi^{p(i)p(x)}v^{-i\cdot |x|} xe_i'(y) \big)\\
&=\bt^{\phi(i,|xy|-i)}\Psi(e_i'(xy)),
\end{align*}
where the equation $(\star)$ follows from the inductive assumption and \eqref{eq:psi}
and $(\star\star)$ follows from  \eqref{eq:phisymmetrized}.

Finally, we prove (3). Such an identity for $\Psi(x)$  follows
by the definition of $\Psi$ and (1), and the claim that this is an $i$-string decomposition
follows from (2).
\end{proof}

\begin{prop}  \label{prop:latticeinv}
The isomorphism $\Psi$ preserves the lattice ${\mathcal L}[\bt]$,
i.e., $\Psi({\mathcal L}[\bt])={\mathcal L}[\bt]$. Furthermore,
$\Psi$ induces an isomorphism $\Psi_0$ on ${\mathcal
L}[\bt]/v{\mathcal L}[\bt]$ such that
\[
\Psi_0(x)=\bt^{\ell(x)}x \qquad \forall x\in \mathcal B,
\]
where $\ell(x)$ is some integer depending on $x$.
\end{prop}

\begin{proof}
We first observe that $\Psi({\mathcal L}[\bt])\subseteq {\mathcal L}[\bt]$,
as this follows from using induction on height along with Lemma~ \ref{lem:latticefacts}(1) and (3),
and Lemma~ \ref{lem:psilattice}(3). On the other hand,
Lemma \ref{lem:psilattice} can be rewritten in terms of $\Psi^{-1}$ (essentially by replacing
$\bt$ with $\bt^{-1}$ in (2) and (3)) and so a similar argument
shows $\Psi^{-1} ({\mathcal L}[\bt])\subseteq {\mathcal L}[\bt]$.
Therefore $\Psi({\mathcal L}[\bt]) ={\mathcal L}[\bt]$.

Let $x+ v{\mathcal L}[\bt] \in \mathcal B$. We proceed by induction on the height of $x$.
First note that $\Psi_0(1+v{\mathcal L}[\bt])=1+v{\mathcal L}[\bt]$
and $\Psi_0(\pi+v{\mathcal L}[\bt])=-\pi+v{\mathcal L}[\bt]$, so the proposition holds with
$\ell(1+v{\mathcal L}[\bt])=0$ and $\ell(\pi+v{\mathcal L}[\bt])=2$.

If $\height |x|\geq 1$,
then by Lemma~ \ref{lem:latticefacts}(2) and (3), there is an $i\in I$
such that we can write $x+ v{\mathcal L}[\bt]=\theta_i^{(n)}x_n+ v{\mathcal L}[\bt]$ with
$x_n+ v{\mathcal L}[\bt] \in \cB$ and $n>0$.
Then by induction on the height and Lemma~\ref{lem:psilattice}(3),
we have
\begin{equation*}\label{eq:psinoughtcrybas}
\Psi_0(x+ v{\mathcal L}[\bt])=\bt^{\phi(ni,\nu)-n^2d_i+\ell(x_n+v{\mathcal L}[\bt]^{\pi})}x+ v{\mathcal L}[\bt].
\end{equation*}
The proposition is proved.
\end{proof}

It was stated in \cite[Proposition 6.7]{CHW2} that ${\mathcal L}$ is
$\varrho$-invariant. In contrast to the non-super setting (as done
by Lusztig and Kashiwara), this is not easy to verify algebraically
using the tools in {\em loc. cit.} because the bilinear form on
${\mathcal L}/v{\mathcal L}$ is not positive definite. Here we are
in a position to furnish an algebraic proof of \cite[Proposition
~6.7]{CHW2}.

\begin{prop}
The involution $\varrho$ preserves ${\mathcal L}$, i.e.,
$\varrho({\mathcal L})={\mathcal L}$.
\end{prop}

\begin{proof}

Since $\frac{1}{2}\in \cA$, we note that

\[\mathcal L=\var_+ \mathcal L\oplus \var_- \mathcal L
\cong \mathcal L|_{\pi=1}\oplus \mathcal L|_{\pi=-1}.\]
We similarly have a decomposition $\varrho=\varrho_+\oplus\varrho_-$
where $\varrho_\pm(x)=\varrho(\var_\pm x)$, and by definition
we see that under the isomorphism
$\var_\pm \schqg\cong \schqg|_{\pi=\pm 1}$, $\varrho_\pm$ corresponds
to $\varrho|_{\pi=\pm 1}$.

Since it is known \cite{K, L93} that $\varrho_{\pi=1}({\mathcal
L}|_{\pi=1})={\mathcal L}|_{\pi=1}$, it suffices to show that
\[
\varrho|_{\pi=-1} ({\mathcal L}|_{\pi=-1})={\mathcal L}|_{\pi=-1}
.\]

Since $\Psi (\pi) =-\pi$, we have $\Psi({\mathcal L}[\bt]|_{\pi=1})={\mathcal L}[\bt]|_{\pi=-1}$.
Let $x\in{\mathcal L}|_{\pi=-1}$.
Since $x\in{\mathcal L}|_{\pi=-1}\subset {\mathcal L}[\bt]|_{\pi=-1}$, by Proposition~\ref{prop:rhopsi} we have
$$
\varrho|_{\pi=-1} (x)
 =(-1)^{\frac{\mathbf{N}(|x|)}{2}+\mathbf{p}(|x|)}
 \Psi\varrho|_{\pi=1} \Psi^{-1}(x)
  \in {\mathcal L}[\bt]|_{\pi=-1}.
$$
On the other hand,
by definition we have $\varrho(x)\in \chqg|_{\pi=-1}$, and hence
$$
\varrho|_{\pi=-1}(x)\in {\mathcal L}[\bt]|_{\pi=-1}\cap \chqg|_{\pi=-1}={\mathcal L}|_{\pi=-1}.
$$
The proposition is proved.
\end{proof}

\subsection{Comparing canonical bases}

The bar involution on $\chqg$ in \eqref{eq:bar} extends trivially to an involution
$\bar{\phantom{c}}$ of $\tpchqg$ and $\schqg$
by letting $\ov{t}=t$ and $\ov{\bt} =\bt$ respectively.

\begin{lem}  \label{lem:bar}
The map $\Psi$ commutes with the bar map on $\schqg$, i.e.,
$\bar{\phantom{c}}\circ\Psi=\Psi\circ\bar{\phantom{c}}$.
\end{lem}

\begin{proof}
By the definition of $\Psi$ given in Theorem~\ref{thm:halfiso},
the only nontrivial thing to check is the commutativity when acting on $v$. Indeed, recalling $\bt^4=1$, we have
\[
\bar{\Psi(v)}=\bt^{-1}\pi v^{-1}=-\pi(\bt^{-1} v)^{-1}=\Psi(\bar{v}).\]
The lemma is proved.
\end{proof}

As proven in \cite{CHW2} (generalizing the approach of \cite{K}),
there exists a globalization map $G:{\mathcal L}[\bt]/v{\mathcal
L}[\bt]\rightarrow {\mathcal L}[\bt]\cap \bar{{\mathcal L}}[\bt]$
such that for each $b\in \cB$, $G(b)$ is the unique bar-invariant
vector in ${\mathcal L}[\bt]$ such that $G(b)+v{\mathcal L}[\bt]=b$.
The  set $\set{G(b):b\in \cB}$ is  called the canonical $\pi$-basis
for $\chqg$.

Specializing $\pi=1$ yields the usual canonical basis of
Lusztig and Kashiwara, while specializing $\pi=-1$ yields a
(signed) canonical basis for the half quantum supergroup.
Even though we have established a connection
on the level of crystal lattices and crystal bases, it is somewhat surprising to see that
$\Psi$ allows us to establish a direct and precise link between
the canonical bases for the two specializations.
Recall $\ell(\cdot)$ from Proposition~\ref{prop:latticeinv}, which
is integer-valued but may not be even-integer-valued in general.

\begin{thm}\label{thm:CB comparison}
For any $b\in \cB$, we have
\[
\Psi(G(b))=\bt^{\ell(b)}G(b).
\]

In particular, $\Psi(G(b)|_{\pi=1})$ is proportional to $G(b)|_{\pi=-1}$.

\end{thm}

\begin{proof}
It follows by Lemma~\ref{lem:bar} that $\Psi(G(b))$ is bar-invariant.
It follows by the definition of the maps and Proposition~ \ref{prop:latticeinv} that
\[\Psi(G(b))+ v{\mathcal L}[\bt] =\Psi(b) =\bt^{\ell(b)} b.\]
Therefore, $\bt^{-\ell(b)}\Psi(G(b))=G(b)$ and thus
$\Psi(\var_+G(b))=\var_{-}\bt^{\ell(b)}G(b)$.
\end{proof}


\begin{example}
Let $(I,\cdot)$ be the super Cartan datum associated to $\mathfrak{osp}(1|4)$
with $I=\set{\mathbf{1},\mathbf{2}}$ (where $\mathbf{1}$ is the odd simple root) and Dynkin diagram given by
$$\qquad\qquad\qquad
\xy
(-25,0)*{\bullet};(-15,0)*{\circ}**\dir{=};
(-20,0)*{<};
(-25,3)*{\mathbf{1}};(-15,3)*{\mathbf{2}}
\endxy
$$
Then
\[p(\mathbf{1})=1,\quad p(\mathbf{2})=0;\]
\[\mathbf{1}\cdot \mathbf{1}=2,\quad \mathbf{1}\cdot \mathbf{2}
=\mathbf{2}\cdot \mathbf{1}=-2,\quad \mathbf{2}\cdot \mathbf{2}=4;\]
\[\phi(\mathbf{1},\mathbf{1})=1,\quad \phi(\mathbf{1},\mathbf{2})= 0,
\quad \phi(\mathbf{2},\mathbf{1})=-2,\quad
\phi(\mathbf{2},\mathbf{2})=2.\]

It is an easy computation that
\[\tf_{\mathbf{1}}\tf_{\mathbf{2}}\tf_{\mathbf{1}} 1
=\theta_{\mathbf{1}}(\theta_{\mathbf{2}}\theta_{\mathbf{1}}-v^2\theta_{\mathbf{1}}\theta_{\mathbf{2}})+
v^2\theta^{(2)}_{\mathbf{1}}\theta_{\mathbf{2}}.\]
In particular, $G(\tf_{\mathbf{1}}\tf_{\mathbf{2}}\tf_{\mathbf{1}} 1+v{\mathcal L})=
\theta_{\mathbf{1}}\theta_{\mathbf{2}}\theta_{\mathbf{1}}$,
and $\Psi \big(G(\tf_{\mathbf{1}}\tf_{\mathbf{2}}\tf_{\mathbf{1}} 1+v{\mathcal L}) \big)
=\bt^{-1}\theta_{\mathbf{1}}\theta_{\mathbf{2}}\theta_{\mathbf{1}}$.
\end{example}

\section{The twistor of modified covering quantum group}
 \label{sec:modified}


\subsection{The modified covering quantum group}

To facilitate the definition of modified covering quantum group
next, we recall the definition of the covering quantum group $\cqg$.
We recall that $b_{ij} = 1- a_{ij}.$

\begin{dfn} \cite{CHW1}
 \label{dfn:cqg}
The covering quantum group
$\cqg$ associated to a super root datum $(Y, X, I, \cdot)$ is the $\Q(v)^{\pi}$-algebra with generators
$E_i, F_i$, $K_\mu$, and $J_\mu$, for  $i\in I$ and $\mu\in Y$, subject to the
relations:
\begin{equation}\label{eq:JKrels}
J_\mu J_\nu=J_{\mu+\nu},\quad K_\mu K_\nu=K_{\mu+\nu},\quad K_0=J_0=J_\nu^2=1,\quad
J_\mu K_\nu=K_\nu J_\mu,
\end{equation}
\begin{equation}\label{eq:Jweightrels}
J_\mu E_i=\pi^{\ang{\mu,i'}} E_i J_\mu,\quad J_\mu F_i=\pi^{-\ang{\mu,i'}} F_i J_\mu,
\end{equation}
\begin{equation}\label{eq:Kweightrels}
K_\mu E_i=v^{\ang{\mu,i'}} E_i K_\mu,\quad K_\mu F_i=v^{-\ang{\mu,i'}} F_i K_\mu,
\end{equation}
\begin{equation}\label{eq:commutatorrelation}
E_iF_j-\pi^{p(i)p(j)}F_jE_i=\delta_{ij}\frac{J_{d_i i}K_{d_i i}-K_{-d_i i}}{\pi_i v_i- v_i^{-1}},
\end{equation}
\begin{equation}\label{eq:Eserrerel}
\sum_{k=0}^{b_{ij}} (-1)^k\pi^{\binom{k}{2}p(i)+kp(i)p(j)}\bbinom{b_{ij}}{k}_{v_i,\pi_i}
 E_i^{b_{ij}-k}E_jE_i^k=0 \;\; (i\neq j),
\end{equation}
\begin{equation}\label{eq:Fserrerel}
\sum_{k=0}^{b_{ij}} (-1)^k\pi^{\binom{k}{2}p(i)+kp(i)p(j)}\bbinom{b_{ij}}{k}_{v_i,\pi_i}
 F_i^{b_{ij}-k}F_jF_i^k=0 \;\; (i\neq j),
\end{equation}
for $i,j\in I$ and $\mu,\nu\in Y$.
\end{dfn}

Again by specialization at $\pi=\pm 1$, we obtain the usual quantum
group $\cqg|_{\pi=1}$ (with extra central elements) and the super
quantum group $\cqg|_{\pi=-1}$.
We extend scalars and set $\tpcqg=\Q(v)[t^{\pm 1}]^\pi\otimes_{\Q(v)^\pi} \cqg$.

We endow $\cqg$ with a $\Z[I]$-grading by setting
\begin{equation}
|E_i|=i,\quad |F_i|=-i,\quad |J_\mu|=|K_\mu|=0,
\end{equation}
and also endow $\cqg$ with a $\Z_2$-grading by setting
\begin{equation}
p(E_i)=p(F_i)=p(i),\quad p(J_\mu)=p(K_\mu)=0.
\end{equation}

The definition of the covering quantum group $\cqg$ is also internally coherent
with the notion of the modified covering quantum group $\cmqg$, which we now introduce.

\begin{dfn}
The {\it modified covering quantum group} $\cmqg$
associated to the root datum $(Y, X, I, \cdot)$ is defined to be
the associative  $\Qvp$-algebra without unit which is
generated by the symbols $1_{\lambda}, E_i1_{\lambda}$ and $F_i1_{\lambda}$, for $\lambda\in X$
and $i \in I$, subject to the relations:
\begin{eqnarray}
 &1_{\lambda} 1_{\lambda'} =\delta_{\lambda, \lambda'} 1_{\lambda},
  \vspace{6pt}\label{eq:modified idemp rel}\\
&   (E_i1_{ \lambda })  1_{\lambda'} =  \delta_{\lambda, \lambda'} E_i1_{ \lambda}, \quad
1_{\lambda'}   (E_i1_{\lambda}) = \delta_{\lambda', \lambda+ i'} E_i1_{\lambda},
     \vspace{6pt}\label{eq:modified E rel}\\
& (F_i1_{ \lambda})  1_{\lambda'}=  \delta_{\lambda, \lambda'} F_i1_{ \lambda}, \quad
1_{\lambda'}  (F_i1_{\lambda})=  \delta_{\lambda', \lambda- i'} F_i1_{\lambda},
\vspace{6pt}\label{eq:modified F rel}\\
 & (E_iF_j-\pi^{p(i)p(j)}F_jE_i)1_{\lambda}=\delta_{ij}\bra{\langle i, \lambda\rangle}_{v_i,\pi_i}1_\lambda,
 \vspace{6pt}
 \label{eq:modified comm rel}\\
 & \sum_{k=0}^{b_{ij}}(-1)^k \pi^{\binom{k}{2}p(i)+kp(i)p(j)}\bbinom{b_{ij}}{k}_{v_i,\pi_i}
 E^{b_{ij}-k}_iE_jE^{k}_i1_{\lambda}=0
 \;\; (i\neq j), 
 \vspace{6pt}
   \label{eq:modified E Serre}\\
 &\sum_{k=0}^{b_{ij}}(-1)^k \pi^{\binom{k}{2}p(i)+kp(i)p(j)}\bbinom{b_{ij}}{k}_{v_i,\pi_i}
  F^{b_{ij}-k}_iF_jF^{k}_i1_{\lambda}=0
  \;\;  (i\neq j), 
   \label{eq:modified F Serre}
\end{eqnarray}
where $i,j \in I$, $\la, \la'\in X$, and we use the notation $xy1_\lambda=(x1_{\lambda+|y|})(y1_\lambda)$
for $x,y\in \cqg$.
\end{dfn}

As with the half covering quantum groups, if we set $\pi=1$ then $\cmqg|_{\pi=1}$
is the modified quantum group of Lusztig, whereas if $\pi=-1$ then $\cmqg|_{\pi=-1}$
is the modified quantum supergroup.

The algebra $\cmqg$ has a (left) $\cqg$-action given by
\[
E_i\cdot x1_\lambda=(E_i1_{\lambda+|x|})x1_\lambda,\quad F_i
 \cdot x1_\lambda=(F_i1_{\lambda+|x|})x1_\lambda,
 \]
\[
K_\nu\cdot x1_\lambda=v^{\ang{\nu,\lambda+|x|}} x1_\lambda,
 \quad J_{\nu}\cdot x1_\lambda=\pi^{\ang{\nu,\lambda+|x|}}x1_{\lambda}.
 \]
 There is also a similar right $\cqg$-action on $\cmqg$.
Then as in \cite{Lu1},   $\cmqg$ can be identified with the $\Qvp$-algebra on the symbols $x1_\lambda$
for $x\in \cqg$ and $\lambda\in X$ satisfying
\begin{equation}\label{eq:altmodifiedrel}
x1_\lambda y1_\mu=\delta_{\lambda,\mu+|y|}xy1_\mu,\quad
K_\nu1_\lambda=v^{\ang{\nu,\lambda}} 1_\lambda,\quad J_{\nu}1_\lambda
=\pi^{\ang{\nu,\lambda}}1_{\lambda}.
\end{equation}

Denote by $_{\Z} \!\cmqg$ the $\Z[v^{\pm 1}]^\pi$-subalgebra of $\cmqg$ generated by
$1_{\lambda}, E_i^{(n)} 1_{\lambda}$ and $F_i^{(n)}1_{\lambda}$, for $n \ge 1$, $\lambda\in X$
and $i \in I$ (here we recall the definition of divided powers \eqref{eq:thetadivpow}).
Then  $_{\Z} \!\cmqg$ is a $\Z[v^{\pm 1}]^\pi$-form of $\cmqg$.

\subsection{The twistor $\dot{\Psi}$}

Recall the bilinear form $\phi(\cdot,\cdot):\Z[I]\times
\Z[I]\rightarrow \Z$ from \eqref{eq:phidef}, and that we have an
embedding $\Z[I]\hookrightarrow X$ given by $i\mapsto i'$. Fix once
and for all a  transversal $C \subset X$ for the coset
representatives of $X/\Z[I]$. Then we define the bilinear pairing
$\dot\phi(\cdot,\cdot):\Z[I]\times X\rightarrow \Z$ by
\begin{equation}
\dot{\phi}(\nu,\mu'+\lambda)=\phi(\nu,\mu),
\quad \ \text{for all }\nu,\mu\in \Z[I], \; \lambda\in C.
\end{equation}

The map $\dot\Psi$ in the following theorem can be viewed as a counterpart in the setting of
modified covering quantum group of  the isomorphism $\Psi$
in Theorem~ \ref{thm:halfiso}. Note that
we do not need to use a twisted multiplication in this setting as for $\Psi$.
By base changes we set as usual $\cmqg[\bt] = \Q(v,\bt)^\pi \otimes_\Qvp \cmqg$ and
$_{\Z} \!\cmqg[\bt] =\Z[v^{\pm 1}, \bt]^\pi \otimes_{\Z[v^{\pm 1}]^\pi} {}_{\Z} \!\cmqg$.

\begin{thm}\label{thm:modauto}
\begin{enumerate}
\item
There is an automorphism
$\dot{\Psi}$ of the $\Q(\bt)$-algebra $\cmqg[\bt]$ of order $4$
such that, for all $i\in I$ and $\la \in X$,
\[
\dot\Psi(1_\lambda)=1_\lambda,\quad \dot\Psi(E_i1_\lambda)
=\bt^{d_i\ang{i,\lambda}-\dot{\phi}(i,\lambda)} E_i1_{\lambda},\quad
\dot\Psi(F_i1_\lambda)=\bt^{\dot{\phi}(i,\lambda)} F_i1_{\lambda},
\]
\[
\dot\Psi(\pi)=-\pi,\quad \dot\Psi(v)=\bt^{-1}v.\]

\item The automorphism $\dot{\Psi}$ preserves the $\Z[v^{\pm 1},\bt]^\pi$-form $_{\Z} \!\cmqg[\bt]$.
\end{enumerate}
\end{thm}

\begin{proof}
(1) Once we verify that the endomorphism $\dot\Psi$ is well defined,
it is clearly an automorphism of order four by checking the images
of the generators. To verify that $\dot\Psi$ is well defined, it
suffices to check that the images of the generators satisfy the
relations. It is straightforward to verify \eqref{eq:modified idemp
rel}-\eqref{eq:modified F rel}, and we leave that as an exercise to
the reader.

Let us check \eqref{eq:modified comm rel}. We compute
\begin{align*}
(& \bt^{d_i\ang{i,\lambda-j'} -\dot{\phi}(i,\lambda-j')} E_i1_{\lambda-j'})(\bt^{\dot{\phi}(j,\lambda)} F_j1_{\lambda})
 \\
 &\; - (-\pi)^{p(i)p(j)}  (\bt^{\dot{\phi}(j,\lambda+i')} F_j1_{\lambda+i'})
(\bt^{d_i\ang{i,\lambda}-\dot{\phi}(i,\lambda)}E_i1_{\lambda})\\
&\quad =\bt^{d_i\ang{i,\lambda-j'}-\dot{\phi}(i,\lambda-j')+\dot{\phi}(j,\lambda)}
(E_iF_j-(-\pi)^{p(i)p(j)}\bt^{i\cdot j + \phi(j,i)-\phi(i,j)}F_jE_i)1_\lambda
\\
&\quad =\bt^{d_i\ang{i,\lambda-j'}-\dot{\phi}(i,\lambda-j')+\dot{\phi}(j,\lambda)}
(E_iF_j-\pi^{p(i)p(j)}F_jE_i)1_\lambda\\
&\quad =\delta_{ij}\bt_i^{\ang{i,\lambda}-1}\bra{\ang{i,\lambda}}_{v_i,\pi_i}\\
&\quad =\delta_{ij}\bra{\ang{i,\lambda}}_{\bt_i^{-1}v_i,-\pi_i}.
\end{align*}

Next, let us check the Serre relations. As the proof of \eqref{eq:modified F Serre} are
similar, we will only check \eqref{eq:modified E Serre}.
Let us set
\[E_{ij}(k)=\dot\Psi(E_i^{b_{ij}-k}E_jE_i^k1_\lambda).\]
We want to verify that
\begin{equation}\label{eq:dotpsimodifiedserre}
\sum_{k=0}^{b_{ij}}(-1)^k (-\pi)^{\binom{k}{2}p(i)
 +kp(i)p(j)}\bbinom{b_{ij}}{k}_{\bt^{-1}v_i,-\pi_i}E_{ij}(k)=0.
\end{equation}

First note that
\[
\Psi(E_i^s1_\lambda)
=\prod_{t=1}^s\bt^{d_i\ang{i,\lambda+(t-1)i'}-\dot{\phi}(i,\lambda+(t-1)i')}E_i1_{\lambda+(t-1)i'}
=\bt^{\binom{s}{2}d_i+s(d_i\ang{i,\lambda}-\dot{\phi}(i,\lambda))}E_i^s1_\lambda.
\]
By using the factorization
$E_i^{b_{ij}-k}E_jE_i^k1_\lambda
=E_i^{b_{ij}-k}1_{\lambda+j'+ki'} E_j1_{\lambda+ki'}E_i^k1_\lambda$
and the identity $\binom{k}{2}+\binom{b_{ij}-k}{2}+k(b_{ij}-k)=\binom{b_{ij}}{2}$, we compute that
\[E_{ij}(k)=\bt^{\spadesuit_{ij}(k)+\heartsuit_{ij}}E^{b_{ij}-k}_iE_jE^{k}_i1_{\lambda},\]
where
\begin{align*}
\heartsuit_{ij}&=b_{ij}(i\cdot j + d_i\ang{i,\lambda}-\dot{\phi}(i,\lambda))
+d_j\ang{j,\lambda}-\dot{\phi}(j,\lambda)+\binom{b_{ij}}{2},
  \\
\spadesuit_{ij}(k)&=-k\phi(j,i)-(b_{ij}-k)\phi(i,j).
\end{align*}

Then
\begin{align*}
\bbinom{b_{ij}}{k}_{\bt^{-1}v_i,-\pi_i}E_{ij}(k)
=\bt^{\heartsuit_{ij}+\spadesuit_{ij}(k)
+k(b_{ij}-k)d_i}\bbinom{b_{ij}}{k}_{v_i,\pi_i}E_i^{b_{ij}-k}E_jE_i^k1_{\lambda}.
\end{align*}

Recall $\clubsuit$ from \eqref{eq:club}.
Since $\phi(i_1,i_2)\in 2\Z$ for $i_1\neq i_2\in I$, we see that
\[
\spadesuit_{ij}(k)+k(b_{ij}-k)d_i
\equiv \clubsuit
\equiv 2\binom{k}{2} + 2k p(i) p(j)+c(i,j)\mod 4.
\]
Then we see that
\begin{align*}
\sum_{k=0}^{b_{ij}}&(-1)^k (-\pi)^{\binom{k}{2}p(i)+kp(i)p(j)}\bbinom{b_{ij}}{k}_{\bt^{-1}v_i,-\pi_i}E_{ij}(k)
\\&=\bt^{\heartsuit_{ij}+c(i,j)}\sum_{k=0}^{b_{ij}}(-1)^k \pi^{\binom{k}{2}p(i)+kp(i)p(j)}
\bbinom{b_{ij}}{k}_{v_i,\pi_i}
 E^{b_{ij}-k}_iE_jE^{k}_i1_{\lambda}=0.
\end{align*}
This finishes the verification of the Serre relations, whence (1).

Part (2) follows immediately from (1) by noting that $\dot\Psi$ preserves the divided powers up to
some integer power of $\bt$.
\end{proof}

Since  $\dot\Psi(\pi) = -\pi$, we obtain the following variant of Theorem~ \ref{thm:modauto}.

\begin{thm}
\label{thm-Udot}
The automorphism $\dot{\Psi}$ of $\cmqg[\bt]$
induces an isomorphism of $\Q(\bt)$-algebras $\cmqg[\bt]|_{\pi=1} \cong \cmqg[\bt]|_{\pi=-1}$
and an isomorphism of $\Z[\bt]$-algebras ${}_\Z\cmqg[\bt]|_{\pi=1}\cong {}_\Z\cmqg[\bt]|_{\pi=-1}$.
In particular,  we have embeddings
$\cmqg|_{\pi=\pm 1}\hookrightarrow \cmqg[\bt]|_{\pi=\mp 1}$ and
${}_\Z\cmqg|_{\pi=\pm 1}\hookrightarrow {}_\Z\cmqg[\bt]|_{\pi=\mp 1}$.
\end{thm}

Hence, one may view the algebras
$\cmqg|_{\pi= 1}$ and $\cmqg|_{\pi=-1}$ as two  different rational forms
of the algebra $\cmqg[\bt]|_{\pi=1}$ (or equivalently, of $\cmqg[\bt]|_{\pi=-1}$).
We shall refer to the automorphism $\dot\Psi$  in Theorem~ \ref{thm:modauto}
as a {\em twistor} on $\cmqg[\bt]$.

\begin{rem}
\begin{enumerate}
\item The definition of the covering quantum group $\cqg$ and the
modified covering quantum group $\cmqg$
 makes sense for super Cartan datum
 without the bar-consistent condition (d) in Definition~\ref{dfn:scd}.
 But the above theorems require the bar-consistent condition.

\item
The integer $\dot{\phi}(i,\lambda)$ admits a geometric interpretation
(compare the integers $e_{\mu, n\alpha_i} $ and $f_{\mu, n\alpha_i}$ in ~\cite[5.1]{Li12}).

%
\end{enumerate}
\end{rem}

\subsection{Category equivalences}

Recall that a $\cmqg$-module $M$ over $\Qv$ is called {\em unital}
if each $m\in M$ is a finite sum of the form $m=\sum_{\lambda\in
X}1_\lambda m$.
When specializing $\pi$ to $\pm 1$, we obtain the definition of a
unital $\cmqg|_{\pi=\pm 1}$-module  over  $\Qv$. We denote the
categories of unital modules over  $\Qv$ of $\cmqg$ (resp.,
$\cmqg|_{\pi=1}$, $\cmqg|_{\pi=- 1}$) by $\dot\bC$ (and resp.,
$\dot\bC_{\pi=1}$, $\dot\bC_{\pi=- 1}$). We have $\dot\bC
=\dot\bC_{\pi=1} \oplus \dot\bC_{\pi=- 1}$. We denote the category
of unital modules  over the field $\Q(v,\bt)$ of $\cmqg$  (and
resp., $\cmqg|_{\pi=1}$, $\cmqg|_{\pi=- 1}$) by $\dot\bC^\bt$ (and
resp., $\dot\bC^\bt_{\pi=1}$, $\dot\bC^\bt_{\pi=- 1}$).

The following  is an immediate consequence
of Theorem \ref{thm-Udot}.

\begin{prop}\label{prop:dotcatequiv}
The twistor $\dot\Psi$ induces a category equivalence  between
$\dot\bC^\bt _{\pi=1}$ and $\dot\bC^\bt _{\pi=-1}$.
\end{prop}

The main novelty in the definition
of the covering quantum group $\cqg$ is the additional generators $J_\mu$ in the Cartan subalgebra,
which lead to a natural formulation of the integral form of $\cqg$ and  weight modules of $\cqg$
(see \cite{CHW1}).
A $\cqg$-module over $\Qv$ (and resp., 
$\cqg[\bt]$-module over $\Q(v,\bt)$)
$M$ is called a {\em weight module} of $\cqg$ (and resp., $\cqg[\bt]$) if
$M=\oplus_{\lambda\in X} M_\lambda$ with
\[
M_\lambda=\set{
m\in M \mid  K_\nu m=v^{\ang{\nu,\lambda}} m,\ J_\nu m=\pi^{\ang{\nu,\lambda}} m,
\ \ \forall\ \nu\in Y}.
\]
We denote the category of weight modules of $\cqg$ (and resp.,
$\cqg[\bt]$) over the respective fields as above by $\bC$ (and
resp.,  $\bC^\bt$). Similarly, we have categories of weight modules
of $\cqg[\bt]|_{\pi=\pm 1}$ over $\Q(v,\bt)$ denoted by
$\bC^\bt_{\pi =\pm 1}$.

One can suitably formulate the BGG category $\mathcal O^\bt$,
$\mathcal O^\bt_{\pi=1}$, $\mathcal O^\bt_{\pi=-1}$
as subcategories of $\bC^\bt$, $\bC^\bt_{\pi=1}$ and $\bC^\bt_{\pi=-1}$, respectively.

Recall the definition of the highest weight $\cqg$-modules $V(\lambda)$ over $\Qv$, for $\la \in X$;
see \cite[Proposition 2.6.5]{CHW1}. Then $V(\lambda)_{\pi=\pm 1}$
is a simple $\cqg|_{\pi=\pm 1}$-module. Let
$X^+ =\{\la \in X \mid \ang{i, \la} \in \Z_{\ge 0}, \forall i\in I\}$ be the set of dominant weights.
Then  $\{V(\lambda)_{\pi=1} | \la \in X^+\}$ and $\{V(\lambda)_{\pi=-1} | \la \in X^+\}$
form a complete list of
pairwise non-isomorphic simple integrable modules of $\cqg|_{\pi=1}$ and of $\cqg|_{\pi=-1}$, respectively.

\begin{prop}
\label{Weight-A}
\begin{enumerate}
\item
The categories $\bC^\bt_{\pi=1}$ and $\bC^\bt_{\pi=-1}$ are equivalent.

\item
The characters of the integrable $\cqg|_{\pi=-1}$-module $V(\lambda)_{\pi=-1}$
and the integrable $\cqg|_{\pi=1}$-module $V(\lambda)_{\pi=1}$ coincide, for each $\la\in X^+$.
\end{enumerate}
\end{prop}

\begin{proof}
The equivalence in (1) follows by a sequence of category equivalences:
$$
\bC^\bt_{\pi=1} \cong \dot\bC^\bt _{\pi=1}
\cong \dot\bC^\bt _{\pi=-1} \cong \bC^\bt_{\pi=-1},
$$
where the second equivalence follows by Proposition~\ref{prop:dotcatequiv},
the first equivalence is easy \cite[\S 23.1.4]{L93}, and
the third equivalence is completely analogous.

The category equivalences above always send each object $M$ to $M$ (where the weight
structure remains unchanged), and moreover, the equivalence from $\bC^\bt_{\pi=1}$ to $\bC^\bt_{\pi=-1}$ sends
$V(\lambda)_{\pi=1}$ to $V(\lambda)_{\pi=-1}$ while preserving weight space
decompositions. The proposition is proved.
\end{proof}

\begin{cor} \label{cor:catO}
The BGG categories $\mathcal O^\bt_{\pi=1}$ of modules over
the quantum group $\cqg[\bt]|_{\pi=1}$ and $\mathcal O^\bt_{\pi=-1}$
over the quantum supergroup $\cqg[\bt] |_{\pi=-1}$ are equivalent via $\dot\Psi$.
\end{cor}

\begin{rem}
Thanks to the isomorphism of the integral forms $_{\mathbb Z} \!\cmqg|_{\pi=1}$ and  $_{\mathbb Z} \!\cmqg|_{\pi=-1}$
in Theorem~\ref{thm-Udot},
a version of category equivalence similar to Propositions ~\ref{Weight-A}
holds when specializing $v$ to be a root of unity.
\end{rem}

\begin{rem}
Proposition~\ref{Weight-A}(2) was stated in \cite{CHW1} without proof, and there has been another proof
given in \cite{KKO13}. A version of Proposition~\ref{Weight-A}(1) on the equivalence of the weight modules
of somewhat different algebras
over $\C(v)^\pi$
also appeared in \cite{KKO13}  with a very different proof.
Note that the notion of weight modules in {\em loc. cit.}  is nonstandard and subtle, and
the algebras formulated therein over $\C(v)$ (or $\C(v)^\pi$)
do not seem to admit rational forms or integral forms or modified forms as ours;
in particular, their formulation does not make sense when $v$ is a root of unity.
\end{rem}

\begin{rem}
Let $X^{\text{ev}} = \{\la \in X \mid \ang{i,\la} \in 2\Z, \forall i
\in I_{\ov{1}} \}$. Denote by $\mathcal O^\bt_{\pi=1,v=1}$ (and
resp., $\mathcal O^\bt_{\pi=-1,v=1}$) the BGG category of
$X^{\text{ev}}$-weighted modules over the Lie algebra (and resp.,
Lie superalgebra) associated to the super root datum $(Y,X, I,
\cdot)$. Using the technique of quantization of Lie bialgebras,
Etingof-Kazhdan \cite{EK} established an equivalence of categories
between $\mathcal O^\bt_{\pi=1,v=1}$ and $\mathcal O^\bt_{\pi=1}$.
As a super analogue, Geer \cite{Ge} similarly established a
quantization of Lie bisuperalgebras (Geer's super analogue was
formulated for the finite type basic Lie superalgebras, but it makes
sense for Kac-Moody  as done by Etingof-Kazhdan.) This leads to an
equivalence of categories between $\mathcal O^\bt_{\pi=-1,v=1}$ and
$\mathcal O^\bt_{\pi=-1}$ (where the restriction to the weights in
$X^{\text{ev}}$ is necessary; see the classification of integrable
modules in \cite{K}). When combining with our category equivalence
in Corollary~\ref{cor:catO}, we obtain an equivalence of highest
weight categories between BGG categories for Lie algebras and
superalgebras. This equivalence provides an irreducible character
formula in $\mathcal O^\bt_{\pi=-1,v=1}$ whenever the corresponding
irreducible module of $\mathcal O^\bt_{\pi=1,v=1}$ admits a solution
of the Kazhdan-Lusztig conjecture (by Beilinson-Bernstein,
Brylinski-Kashiwara, Kashiwara-Tanisaki).
\begin{center}
\begin{tikzpicture}[scale=1]
\draw[->] (-1.6,1.) -- (1.4,1.) node[midway,above] {$\dot\Psi$};
\draw (-2,1) node {$\mathcal O^\bt_{\pi=1}$};
\draw (2.3,1) node {$\mathcal O^\bt_{\pi=-1}$};
\draw (-2,-1) node {$\mathcal O^\bt_{\pi=1,v=1}$};
\draw (2.3,-1) node {$\mathcal O^\bt_{\pi=-1,v=1}$};

\draw (2.1,0) node {G};
\draw (-2.3,0) node {EK};

\draw[->] (-1.9,-.6) -- (-1.9,.6);
\draw[->] (2.4,-.6) -- (2.4,.6);
\draw[snake=snake,-] (-.9,-1) -- (.9,-1);
\draw[<-] (-1.2,-1) -- (-.9,-1);
\draw[<-] (1.4,-1) -- (.9,-1);
\end{tikzpicture}
\end{center}
\end{rem}


\subsection{Extended covering quantum groups}

We first work over a formal
parameter $t$.
Let $\mathbf{T}$ be the group algebra (in multiplicative form) of the group
$\Z[I]\times Y$, that is, the $\Q(v)[t^{\pm 1}]^\pi$-algebra with generators
$T_{\mu}, \Upsilon_{\nu}$, for $\mu\in Y$ and $\nu\in \Z[I]$, and relations
\begin{equation}\label{eq:Trelns}
T_\mu T_{\mu'}=T_{\mu+\mu'},\quad \Upsilon_\nu \Upsilon_{\nu'}=\Upsilon_{\nu+\nu'},
\quad T_\mu \Upsilon_\nu=\Upsilon_\nu T_\mu, \quad T_0=\Upsilon_0=1.
\end{equation}
We define an action of $\mathbf{T}$ on $\tpcqg$ by
\begin{equation}
T_\mu\cdot x=t^{\ang{\mu,\eta'}} x,\quad \Upsilon_\nu\cdot x=t^{\phi(\nu,\eta)} x\text{ for all }x\in \tpcqg_\eta.
\end{equation}
Then we form the semi-direct $\Q(v)[t^{\pm 1}]^\pi$-algebra
$\widehat\cqg[t^{\pm 1}]=\mathbf{T} \ltimes \tpcqg$ with respect to the above action of $\mathbf T$;
that is, $TxT^{-1}=T\cdot x$ for all $T\in \mathbf{T}$ and $x \in \tpcqg$.
By specialization, we obtain a $\Q(v,\bt)^\pi$-algebra
$\widehat\cqg[\bt]$, which is called the  {\em extended covering quantum group}.

 \begin{prop}\label{prop:qgiso}
There is a $\Q(\bt)$-algebra automorphism $\widehat{\Psi}$ on $\widehat\cqg[\bt]$
such that
\[
\widehat{\Psi}(E_i)=\bt_i^{-1}\Upsilon_i^{-1} T_{d_i i}E_i,\quad
\widehat{\Psi}(F_i)
=F_i\Upsilon_i,\quad
\widehat{\Psi}(K_\nu)=T_{-\nu} K_\nu,\quad
\widehat{\Psi}(J_\nu)=T_\nu^{2} J_\nu,\]
\[
\widehat{\Psi}(T_\nu)=T_\nu,\quad
\widehat{\Psi}(\Upsilon_\nu)=\Upsilon_\nu, \quad
\widehat{\Psi}(v)=\bt^{-1} v, \quad \widehat{\Psi}(\pi)=-\pi.\]
\end{prop}
The automorphism $\widehat{\Psi}$  will be
called the  {\em twistor} on $\widehat\cqg[\bt]$.

\begin{proof}
We first show that such a map is well defined by showing that relations
\eqref{eq:JKrels}-\eqref{eq:Fserrerel} and \eqref{eq:Trelns} are satisfied
by the images of the generators. The relations
\eqref{eq:JKrels}-\eqref{eq:Kweightrels}
and \eqref{eq:Trelns} are straightforward to verify,
and we leave this to the reader.

Let us verify \eqref{eq:commutatorrelation}. On one hand, we have
\begin{align}
\notag \widehat{\Psi}(E_i) & \widehat{\Psi}(F_j)-\widehat{\Psi}(\pi)^{p(i)p(j)}\widehat{\Psi}(F_j)\widehat{\Psi}(E_i)\\
\notag &=\bt_i^{-1}\Upsilon_i^{-1} T_{d_i i}E_iF_j\Upsilon_j
  -(-\pi)^{p(i)p(j)}F_j\Upsilon_j\bt^{-1}_i\Upsilon_i^{-1} T_{d_i i}E_i\\
\notag &=\bt^{-d_i+d_j-\phi(j,i)}\Upsilon_i^{-1}\Upsilon_jT_{d_i i}
  \Big( E_iF_j-\bt^{i\cdot j+\phi(j,i)-\phi(i,j)}(-\pi)^{p(i)p(j)}F_jE_i \Big)\\
\label{eq:PsicommLHS}&=\bt^{-d_i+d_j-\phi(j,i)}\Upsilon_i\Upsilon_j^{-1}T_{d_i i}(E_iF_j-\pi^{p(i)p(j)}F_jE_i),
\end{align}
where the last equality follows from \eqref{eq:phisymmetrized} and $\bt^2=-1$.
On the other hand,
\begin{align}
\notag \delta_{ij}\frac{\widehat{\Psi}(J_{d_i i})
\widehat{\Psi}(K_{d_i i})-\widehat{\Psi}(K_{-d_i i})}{\widehat{\Psi}(\pi_i)\widehat{\Psi}(v_i)-\widehat{\Psi}(v_i)^{-1}}
&=\delta_{ij}\frac{T_{d_i i} J_{d_i i}K_{d_i i}-T_{d_i i} K_{-d_i i}}{(-\bt)^{-d_i} \pi_iv_i-\bt_i v_i^{-1}}\\
\label{eq:PsicommRHS}&=\delta_{ij}
\bt_i^{-1}T_{d_i i}\frac{ J_{d_i i}K_{d_i i}-K_{-d_i i}}{\pi_i v_i-v_i^{-1}}.
\end{align}
Then comparing \eqref{eq:PsicommLHS} and \eqref{eq:PsicommRHS}, we see that they
are equal for all $i,j\in I$,whence \eqref{eq:commutatorrelation}.

It remains to check the Serre relations \eqref{eq:Eserrerel} and
\eqref{eq:Fserrerel}. As these computations are entirely similar,
let us prove \eqref{eq:Fserrerel}.
Then recalling \eqref{eq:club}, we see that
\[(F_i\Upsilon_i)^{b_{ij}-k}(F_j\Upsilon_j)(F_i\Upsilon_i)^{k}
=\bt^{\binom{b_{ij}}{2}-k(b_{ij}-k)d_i+\clubsuit}F_i^{b_{ij}-k}F_jF_i^k\Upsilon_{b_{ij}i+j}.
\]
Hence as in the proof of Theorem~ \ref{thm:halfiso}, we have
\begin{align*}
&\sum_{k=0}^{b_{ij}} (-1)^k(-\pi)^{\binom{k}{2}p(i)+kp(i)p(j)}\bbinom{b_{ij}}{k}_{\bt^{-1}v_i,-\pi_i}
(F_i\Upsilon_i)^{b_{ij}-k}(F_j\Upsilon_j)(F_i\Upsilon_i)^k\\&
=\parens{\sum_{k=0}^{b_{ij}} (-1)^k\pi^{\binom{k}{2}p(i)+kp(i)p(j)}\bbinom{b_{ij}}{k}_{v_i,\pi_i}
F_i^{b_{ij}-k}F_jF_i^k}\bt^{\binom{b_{ij}}{2}+c(i,j)}\Upsilon_{b_{ij}i+j}=0.
\end{align*}
The proposition is proved.
\end{proof}

\begin{rem}  \label{rem:UUdot}
Here is a heuristic way of thinking about the extended covering quantum group and its twistor.
The algebra $\cqg$ acts on $\cmqg$ via
\[1\mapsto \sum_{\lambda\in X} 1_\lambda,\; E_i\mapsto \sum_{\lambda\in X} E_i 1_\lambda,\;
F_i\mapsto \sum_{\lambda\in X} F_i 1_\lambda,\;
K_\nu\mapsto \sum_{\lambda\in X} v^{\ang{\nu,\lambda}}1_\lambda,\;
J_\nu\mapsto \sum_{\lambda\in X} \pi^{\ang{\nu,\lambda}}1_\lambda.\]
Then $\dot\Psi$ induces an alternate
$\cqg$-module structure on $\cmqg$ via
\begin{align*}
\notag1\mapsto \sum &1_\lambda,\quad E_i\mapsto \sum_{\lambda\in X}
\bt^{d_i \ang{i,\lambda}-\dot{\phi}(i,\lambda)}E_i 1_\lambda,\quad
F_i\mapsto \sum_{\lambda\in X} \bt^{\dot{\phi}(i,\lambda)}F_i 1_\lambda,\\
&\label{eq:new action on dot U}
K_\nu\mapsto \sum_{\lambda\in X} (\bt^{-1}v)^{\ang{\nu,\lambda}}1_\lambda,\quad
J_\nu\mapsto \sum_{\lambda\in X} (-\pi)^{\ang{\nu,\lambda}}1_\lambda.
\end{align*}
Merging these two actions
leads to the introduction of new semisimple elements $T_\nu$ and $\Upsilon_\mu$ such that
$
 T_\nu\mapsto \sum_{\lambda\in X} t^{\ang{\nu,\lambda}}1_\lambda
$
and
$\Upsilon_\mu\mapsto \sum_{\lambda\in X}t^{\dot{\phi}(\mu,\lambda)} 1_\lambda.
$
\end{rem}

\begin{rem}
Some construction similar to the twistor $\widehat{\Psi}$ as in Propoisition~\ref{prop:qgiso} appeared in \cite{KKO13}.
In contrast to {\em loc. cit.}, our formula for $\widehat{\Psi}$ is very explicit.
\end{rem}

By specialization, the twistor on $\widehat\cqg[\bt]$
leads to an isomorphism between the extended super and non-super quantum groups.

\begin{cor}\label{cor:isoquantumgrps}
The $\Q(\bt)$-algebras $\widehat\cqg[\bt] |_{\pi=1}$ and $\widehat\cqg[\bt]_{\pi=-1}$
are isomorphic under $\widehat{\Psi}$.
\end{cor}
%
%


The twistor $\Psi: \chqg[\bt] \rightarrow (\chqg [\bt], *)$ in Theorem~\ref{thm:halfiso}
is  intimately related to the twistor $\widehat{\Psi}: \widehat \cqg[\bt] \rightarrow \widehat \cqg[\bt]$
in Proposition~\ref{prop:qgiso}, as we shall describe.

There is an injective $\Q(v)[t^{\pm 1}]^\pi$-algebra homomorphism (see \cite[\S 2.1]{CHW1})
\begin{equation}\label{eq:minusmap}
(\cdot)^-: \tpchqg \longrightarrow \tpcqg,
\end{equation}
such that $\theta_i^-=F_i$ for all $i\in I.$

\begin{lem}
There is an injective $\Q(v)[t^{\pm 1}]^\pi$-algebra homomorphism
\begin{equation*}
\chi: (\tpchqg, *) \longrightarrow \widehat\cqg[t^{\pm 1}]
\end{equation*}
such that
$$
\chi(x) = x^-\Upsilon_{\nu}, \qquad \forall x \in \tpchqg_\nu.
$$
\end{lem}

\begin{proof}
One checks by definition that, for $x,y\in \tpchqg$ homogeneous,
\[
(x*y)^-\Upsilon_{|x|+|y|}=x^-\Upsilon_{|x|}
 y^-\Upsilon_{|y|}.
 \]
 The lemma is proved.
\end{proof}

Now specializing $t$ to $\bt$ for $\chi$ and $(\cdot)^-$ above, we obtain
an injective $\Q(v, \bt)^\pi$-algebra homomorphism
$
\chi: (\schqg, *) \longrightarrow \widehat\cqg[\bt]$,
and an injective $\Q(v,\bt)^\pi$-algebra homomorphism
$(\cdot)^-: \schqg \longrightarrow \cqg[\bt]$.
The following proposition can be verified by definitions, which we leave to the reader.

\begin{prop}
We have a
commutative diagram of $\Q(\bt)$-algebra homomorphisms:
\begin{center}
\begin{tikzpicture}
\draw (.2,0) node{$\schqg$};
\draw (0,-1.7) node{$(\schqg,*)$};
\draw (2.5,0) node{$\cqg[\bt]$};
\draw (2.5,-1.7) node{$\widehat \cqg[\bt]$};
\draw[right hook-latex,thick] (.7,0)--(2,0) node[midway,above]{$(\cdot)^-$};
\draw[right hook-latex,thick] (.7,-1.7)--(2,-1.7) node[midway,above]{$\chi$};
\draw[-latex,thick] (.2,-.3) -- (.2,-1.4)node[midway,left]{$\Psi$};
\draw[-latex,thick] (2.5,-.3) -- (2.5,-1.4)node[midway,right]{$\widehat{\Psi}$};
\end{tikzpicture}
\end{center}
\end{prop}


\end{document}